\documentclass[12pt]{article}
\usepackage{latexsym}
\usepackage{multirow}
\usepackage{graphicx}
\usepackage{float}

\usepackage{authblk}
\usepackage{multirow}
\usepackage{kantlipsum}
\usepackage{setspace}
\usepackage{xcolor, color}
\def\q{\hfill\rule{1ex}{1ex}}
\def\0{\emptyset}

\def\p{{\bf Proof.~~}}

\usepackage{indentfirst}
\usepackage{amsmath}

\def\0{\emptyset}

\def\p{{\bf Proof.} \quad}
\def\q{\hfill\rule{1ex}{1ex}}

\usepackage{mathrsfs}
\usepackage{amssymb}
\usepackage{listings}
\usepackage{amsmath}
\usepackage{pict2e}
\usepackage{amsfonts}
\usepackage{graphicx}
\usepackage[small]{caption}
\usepackage{subfigure}
\usepackage{mathtools}
\usepackage{multirow}
\usepackage{tikz}
\newtheorem{theorem}{Theorem}[section]

\newtheorem{prob}[theorem]{Problem}	\newtheorem{lemma}[theorem]{Lemma}

\newtheorem{proposition}[theorem]{Proposition}

\newcommand\eqn[2]
{
	\begin{eqnarray} \label{#1}
		#2
	\end{eqnarray}
}

\newcounter{countclaim}

\begin{document}
	\title{Extremal number of edges in graphs without homeomorphically irreducible spanning trees\thanks{This research was partially supported by NSFC (No. 12371345 and 12471325). }}

	\author[1,2]{\small Yibo Li \thanks{Email: liyibo@stu.hubu.edu.cn }}
	
	\author[1,2]{\small Huiqing Liu \thanks{Email: hqliu@hubu.edu.cn}}
	
	\author[3]{\small Xiaolan Hu\thanks{Email: xlhu@ccnu.edu.cn}}
	
	\affil[1]{\footnotesize Hubei Key Laboratory of Applied Mathematics, Faculty of Mathematics and Statistics,
		
		Hubei University, Wuhan 430062, China}
	
	\affil[2]{\footnotesize Key Laboratory of Intelligent Sensing System and Security (Hubei University), Ministry of Education}
	
    \affil[3]{\footnotesize School of Mathematics and Statistics \& Hubei Key Laboratory of Mathematical Sciences,
    	
    Central China Normal University, Wuhan 430079, PR China}
    	
	\date{}
	
	\maketitle
	
	\begin{abstract}
		\baselineskip=0.5cm
		For integers $k\ge 1$ and $n\ge k+1$, let $\operatorname{ex}^{\mathrm{HIST}}_k(n)$ denote the maximum number of edges in a $k$-connected graph of order $n$ which contains no homeomorphically irreducible spanning tree (or briefly HIST). We determine these extremal numbers for $k=1$ and $k=2$. More precisely, we prove that
		$\operatorname{ex}^{\mathrm{HIST}}_1(n)=\binom{n-2}{2}+2$
		for $n\ge 9$, with $L_n$ as the unique extremal graph, and that
		$\operatorname{ex}^{\mathrm{HIST}}_2(n)=\binom{n-3}{2}+4$
		for $n\ge 13$, with $B_n$ as the unique extremal graph. This provides a Tur\'an-type extremal result for spanning trees with no vertices of degree two.
		
		\vskip 0.2cm
		
		{\bf AMS} classification: 05C05, 05C35, 05C50
		
		\vskip.2cm
		
	\end{abstract}
	\vskip 0.2cm
	
	
	
	\section{Introduction}
	All graphs considered in this paper are simple and finite. Let $G=(V(G),E(G))$ be a graph with $|G|=|V(G)|$ and $e(G)=|E(G)|$. For $v\in V(G)$, $N_G(v)$ is the {\em set of neighbors} of $v$ in $G$, $N_G[v]:=N_G(v)\cup\{v\}$, and $d_G(v):=|N_G(v)|$ denotes the {\em degree} of $v$ in $G$. We may omit the index $G$ if there is no risk of confusion. Let $\delta(G)$ and $\Delta(G)$ be the {\em minimum degree} and the {\em maximum degree} of $G$, respectively. A {\em leaf} of $G$ is a vertex of degree $1$ in $G$.

	For any non-empty subset $S\subseteq V(G)$, let $G[S]$ denote the subgraph of $G$ induced by $S$, and write $N_S(v)$ and $d_S(v)$ for $N_G(v)\cap S$ and $|N_G(v)\cap S|$, respectively, for each  $v\in V(G)$. We write $N_G(S)=\bigcup_{v\in S}N_G(v)$. Let $G-S$ denote the subgraph of $G$ induced by $V(G)\backslash S$. If $S=\{v\}$, then we simplify $G-\{v\}$ to $G-v$.
	
	Given two disjoint vertex sets $X$ and $Y$ of $G$, let $E(X,Y)$ be the set of all edges with one end in $X$ and one end in $Y$, and let $e(X,Y)=|E(X,Y)|$. If $X=\{x\}$, we simply write $E(x,Y)$ for $E(X,Y)$. For a subgraph $H$ of $G$, we consider it as both a subgraph and a vertex set of $G$. A subset $C$ of $V(G)$ is called a {\em clique} if every two vertices in $C$ are adjacent. By $K_n$ and $P_n$, we denote a {\em complete graph} and a {\em path} of order $n$, respectively.

	For a connected graph $G$, a spanning tree $T$ of $G$ is called a {\em homeomorphically irreducible spanning tree} (or briefly HIST) of $G$ if $T$ has no vertices of degree $2$. The study of HISTs has attracted
	considerable attention, and several related results have been obtained (see \cite{ABHT90}, \cite{ChRS12}-\cite{DFST15}, \cite{FuST25}-\cite{Malk79}, \cite{ZhWH19} for example). In particular, Ito and Tsuchiya \cite{ItTs22} provided a degree sum condition for the existence of a HIST.
	
	\begin{theorem} \label{Thm1.1}\cite{ItTs22}
		Let $G$ be a graph of order $n\geq8$. If $d(u)+d(v)\geq n-1$
		for any nonadjacent vertices $u$ and $v$ of $G$, then $G$ has a HIST.
	\end{theorem}
	
	Furuya, Saito, and Tsuchiya~\cite{FuST25} established a minimum degree condition for the existence of a HIST.
	
	\begin{theorem}\cite{FuST25}\label{Thm1.2}
		Let $G$ be a connected graph of order $n$. If $\delta(G)\ge4\sqrt{n}$,
		then $G$ has a HIST.
	\end{theorem}
	
	Recently, Li et al. \cite{LDLH26} gave a neighborhood union condition for the existence of a HIST.
	
	\begin{theorem}[\cite{LDLH26}] \label{Thm1.3}
		Let $G$ be a connected graph of order $n\geq270$. If
		$|N(u)\cup N(v)|\ge\frac{n-1}{2}$ for any nonadjacent vertices $u$ and $v$ of $G$, then $G$ has a HIST
		if and only if $G$ does not belong to the four exceptional families of graphs.
	\end{theorem}

In 1941, Tur\'an \cite{Turan} determined the maximum number of edges in an $n$-vertex graph containing no $K_{r+1}$ and characterized the
extremal graphs. This classical result has motivated many Tur\'an-type problems. For example, Ore~\cite{Ore1961} proved that every graph $G$ of order $n$ with $e(G)>\binom{n-1}{2}+1$ is Hamiltonian. Subsequently, many edge conditions for the existence of spanning subgraphs have been obtained (see \cite{Bondy1971,Bondy1972,Erdos1962,Ore1961} for example). 

We call a graph {\em HIST-free} if it contains no HIST. Motivated by the above edge extremal results, we consider the following natural problem.

\begin{prob}
	What is the maximum number of edges in a \(k\)-connected HIST-free graph of order \(n\)?
\end{prob}

Let $\kappa(G)$ denote the vertex-connectivity of $G$. For integers $k\ge 1$ and $n\ge k+1$, define
$$
\operatorname{ex}^{\mathrm{HIST}}_k(n)
:=\max\{e(G): |G|=n,\ \kappa(G)\ge k,\ G\text{ is HIST-free}\}.
$$

In this paper, we determine
$\operatorname{ex}^{\mathrm{HIST}}_k(n)$ for $k=1,2$ and characterize
the corresponding extremal graphs, respectively.

We now introduce the two extremal graphs \(L_n\) and \(B_n\). Let $L_n$ be the graph obtained from $K_{n-2}$ and $P_2$ by joining one leaf of $P_2$ to one vertex of $K_{n-2}$ (see Figure~1(a)). Let $B_n$ be the graph obtained from $K_{n-3}$ and $P_3$ by joining the two leaves of $P_3$ to two distinct vertices of $K_{n-3}$ (see Figure~1(b)).

\medskip

	\begin{center}\label{Fig1}
		\begin{tikzpicture}[scale=0.8]
			\tikzstyle{every node}=[font=\normalsize,scale=0.89]
			\draw (-7,5)node[above=10pt]{$K_{n-2}$};
			\filldraw (-1.8,6)circle(0.8ex);
			\filldraw (-3.5,6)circle(0.8ex);
			\filldraw (-5.2,6)circle(0.8ex);
			\draw[very thick] (-7,6)circle(1.8);
			\draw[very thick] (-5.2,6)--(-3.5,6);
			\draw[very thick] (-3.5,6)--(-1.8,6);
			
			\draw (-5.5,4) node[]{(a)~~~~$L_n$};
			~~~~~~~~~~~~~~~~~~~~~~~~~~~~~~~~~~~~~~~~~~~~~~~~~~
			\draw[very thick] (3,6)circle(1.8);
			\filldraw (4.4,7.1)circle(0.8ex);
			\filldraw (4.4,4.9)circle(0.8ex);
			\draw (3,5)node[above=10pt]{$K_{n-3}$};
			
			\filldraw (6.2,7.1)circle(0.8ex);
			\filldraw (6.2,4.9)circle(0.8ex);
			\filldraw (7.5,6)circle(0.8ex);
			\draw[very thick] (4.4,7.1)--(6.2,7.1);
			\draw[very thick] (6.2,7.1)--(7.5,6);
			\draw[very thick] (4.5,4.9)--(6.2,4.9);
			\draw[very thick] (6.2,4.9)--(7.5,6);
			
			\draw (4,4) node[]{(b)~~~~$B_n$};
			
			\draw (-0.5,3) node[below=0pt]{ Figure 1.~~The extremal graphs: (a) $L_n$ and (b) $B_n$};
			
		\end{tikzpicture}
	\end{center}

\medskip

\noindent{\bf Lower-bound examples.} We first observe that the graphs $L_n$ and $B_n$ yield the corresponding lower bounds. If a graph $G$ has a HIST, then $G$ cannot contain a vertex cut consisting entirely of vertices of degree $2$. It is easy to verify that both $L_n$ and $B_n$ contain such vertex cuts. Therefore, they are HIST-free. Since $L_n$ is connected and $e(L_n)=\binom{n-2}{2}+2$, we have
$$
\operatorname{ex}^{\mathrm{HIST}}_1(n)\ge e(L_n)=\binom{n-2}{2}+2.
$$
Similarly, since $B_n$ is $2$-connected and $e(B_n)=\binom{n-3}{2}+4$, we have $$
\operatorname{ex}^{\mathrm{HIST}}_2(n)\ge e(B_n)=\binom{n-3}{2}+4.
$$
The following two theorems show that these lower bounds are sharp and  characterize the extremal graphs.
	
\begin{theorem}\label{Thm1.4}
			Let $n\ge9$. Then
$$
\operatorname{ex}^{\mathrm{HIST}}_1(n)=\binom{n-2}{2}+2.
$$
Moreover, the unique extremal graph is $L_n$.
\end{theorem}

\begin{theorem}\label{Thm1.5}
	Let $n\ge13$. Then
	$$
	\operatorname{ex}^{\mathrm{HIST}}_2(n)=\binom{n-3}{2}+4.
	$$
	Moreover, the unique extremal graph is $B_n$.
\end{theorem}


\noindent{\bf Remark.} The above two extremal results can be reformulated in terms of edge conditions for the existence of a HIST. More precisely, every connected graph $G$ of order $n\ge 9$ with $e(G)\ge \binom{n-2}{2}+2$ contains a HIST unless $G\cong L_n$, and every $2$-connected graph $G$ of order $n\ge 13$ with $e(G)\ge \binom{n-3}{2}+4$ contains a HIST unless $G\cong B_n$.


The remainder of this paper is organized as follows. Section 2 contains necessary preliminaries. In Section 3, we prove Theorem \ref{Thm1.4}, characterizing the extremal graphs for 1-connected HIST-free graphs. Section 4 is devoted to the proof of Theorem \ref{Thm1.5} for 2-connected graphs. Finally, Section 5 concludes with some remarks and open problems.

	\section{Preliminaries}
	In this section, we provide some lemmas that will be used in the proofs of main results. We begin with basic properties of HITs.
	
	\subsection{Properties of HITs in graphs}
	Let $S,U\subseteq V(G)$, and let $s,s'\in S$ be distinct. If $N_U(s)\cap N_U(s')\neq\emptyset$, then $(s,s')$ is called a {\em good pair} with respect to $U$, and $s'$ is called a {\em twin} of $s$ with respect to $U$.

A subtree $T$ of $G$ is called a {\em homeomorphically irreducible tree} (HIT) if $T$ has no vertex of degree $2$.  Thus, a HIST is precisely a HIT containing all vertices of $G$. In a HIT $T$, each non-leaf vertex is called a {\em stem}. A stem $v$ of $T$ is {\em nice} if $d_T(v)\ge4$. A HIT $T$ is called a {\em $(1,2)$-extendable HIT} if, for every $v\in V(G)\setminus V(T)$, either $v$ is adjacent to a stem of $T$, or $v$ has a twin $v'\in V(G)\setminus V(T)$ with respect to $V(T)$. The following lemma is motivated by Lemma 5 of \cite{ItTs22}.

\begin{lemma}\label{le2.1}
	If $G$ has a (1,2)-extendable HIT, then $G$ has a HIST.
\end{lemma}

Next we give a useful tool which plays a key role in the proof of main results. We omit its proof since it is straightforward.
	\begin{lemma}\label{le2.2}
		Let $T$ be a HIT of $G$ with a nice stem $u$. Suppose that $w\in V(G)\setminus V(T)$ and $w$ is adjacent to some leaf $w^*$ of $T$. Denote $N_T(w^*)=\{v\}$. If there exists some vertex $u^*\in N_T(u)$ such that $u^*w^*\in E(G)\setminus E(T)$, then $T_{+w}:=T+\{ww^*,w^*u^*\}-uu^*$ is a HIT of $G$ with $V(T_{+w})=V(T)\cup\{w\}$ (see Figure 2).
	\end{lemma}

	\begin{center}\label{Fig2}
	\begin{tikzpicture}[scale=0.6]
		\tikzstyle{every node}=[font=\large,scale=0.8]
		\draw (-3.5,-4.3) node[right=2pt]{$v(u)$};
		\draw (-1,-8.2) node[right=2pt]{$w$};
		\draw (-2.3,-6.3) node[right=0pt]{$w^*$};
		\draw (-3.6,-6.3) node[below=0pt]{$u^*$};			
		\filldraw (-2.5,-2)circle(0.8ex);
		\filldraw (-5.5,-4.0)circle(0.8ex);
		\filldraw (-3.6,-4.0)circle(0.8ex);
		\draw (-3.6,-4.0)circle(1.6ex);
		\draw[red] (-1,-8.2)circle(1.6ex);
		\filldraw (-1.6,-4.0)circle(0.8ex);
		\filldraw (0.5,-4.0)circle(0.8ex);	
		\filldraw (-0.9,-4)circle(0.2ex);
		\filldraw (-0.6,-4)circle(0.2ex);
		\filldraw (-0.3,-4)circle(0.2ex);
		\filldraw (-7.2,-6.3)circle(0.8ex);
		\filldraw(-5.8,-6.3)circle(0.8ex);
		\filldraw (-4.9,-6.3)circle(0.8ex);
		\filldraw (-3.6,-6.3)circle(0.8ex);
		\filldraw (-2.3,-6.3)circle(0.8ex);
		\filldraw (-1,-8.2)circle(0.8ex);
		\draw[very thick] (-2.5,-2)--(-5.5,-4);		
		\draw[very thick] (-2.5,-2)--(-1.6,-4);	
		\draw[very thick] (-2.5,-2)--(0.5,-4);	
		\draw[very thick] (-3.6,-4.0)--(-2.5,-2.0);	
		\draw[very thick] (-7.2,-6.3)--(-5.5,-4);	
		\draw[very thick] (-5.8,-6.3)--(-5.5,-4);		
		\draw[very thick] (-4.9,-6.3)--(-3.6,-4);	
		\draw[very thick,dashed,red] (-3.6,-6.3)--(-3.6,-4);	
		\draw[very thick] (-2.3,-6.3)--(-3.6,-4);	
		\draw[very thick,red] (-1,-8.2)--(-2.3,-6.3);	
		\draw[very thick,red] (-2.3,-6.3)--(-3.6,-6.3);		
		\draw (-2.5,-9) node[below=0pt]{(a)~~~$u=v$};
		~~~~~~~~~~~~~~~~~~~~~~~~~~~~~~~~~~~~~~~~~~~~~~~~~~~~~
		\draw (9.5,-2) node[above=5pt]{$u$};		
		\draw (8.8,-3.9) node[below=0pt]{$v$};
		\draw (10.5,-8.2) node[right=2pt]{$w$};
		\draw (9.2,-6.3) node[right=0pt]{$w^*$};
		\draw (10.6,-3.9) node[below=0pt]{$u^*$};			
		\filldraw(9.5,-2)circle(0.8ex);
		\draw(9.5,-2)circle(1.6ex);
		\filldraw (6.5,-4.0)circle(0.8ex);
		\filldraw (8.4,-4.0)circle(0.8ex);
		\filldraw (10.4,-4.0)circle(0.8ex);
		\filldraw (12.5,-4.0)circle(0.8ex);
		\draw[red] (10.5,-8.2)circle(1.6ex);
		\filldraw (11.1,-4)circle(0.2ex);
		\filldraw (11.4,-4)circle(0.2ex);
		\filldraw (11.7,-4)circle(0.2ex);
		\filldraw (5.5,-6.3)circle(0.8ex);
		\filldraw(7.0,-6.3)circle(0.8ex);
		\filldraw (7.7,-6.3)circle(0.8ex);
		\filldraw (9.2,-6.3)circle(0.8ex);
		\filldraw (10.5,-8.2)circle(0.8ex);
		\draw[very thick] (9.5,-2)--(6.5,-4);	
		\draw[very thick] (9.5,-2)--(8.4,-4);		
		\draw[very thick,dashed,red] (9.5,-2)--(10.4,-4);	
		\draw[very thick] (9.5,-2)--(12.5,-4);	
		\draw[very thick,red] (9.2,-6.3)--(10.4,-4);	
		\draw[very thick] (5.5,-6.3)--(6.5,-4);	
		\draw[very thick] (7,-6.3)--(6.5,-4);		
		\draw[very thick] (7.7,-6.3)--(8.4,-4);	
		\draw[very thick] (9.2,-6.3)--(8.4,-4);	
		\draw[very thick,red] (10.5,-8.2)--(9.2,-6.3);	
		\draw[very thick,red] (-2.3,-6.3)--(-3.6,-6.3);	
		\draw (9.5,-9) node[below=0pt]{(b)~~~$u\not=v$};
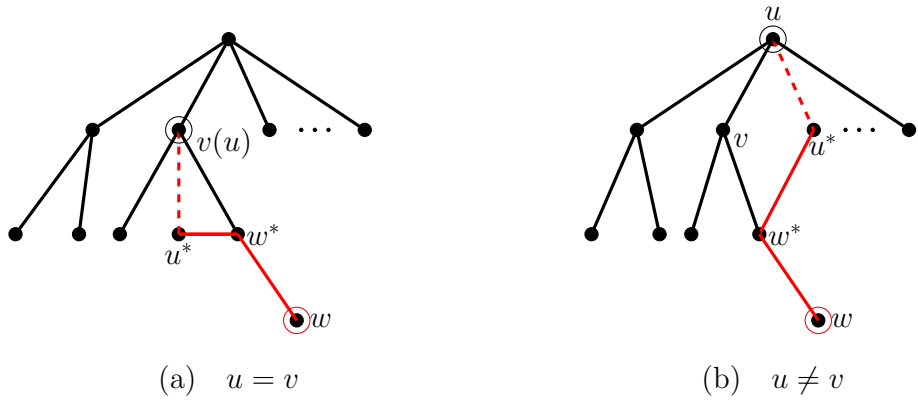
		\draw (3.8,-9) node[below=30pt]{\large Figure~2.~~Construction of a HIT of $G$ (strong lines denote edges of HIT)};
	\end{tikzpicture}
\end{center}

\subsection{Graphs $G$ with $e(G)\ge{\binom{n-3}{2}+4}$}

In this subsection, we will mainly provide some sufficient conditions for the existence of a HIST in graphs satisfying $e(G)\ge{{n-3}\choose2}+4$. For any integer $a$, let $[a]$ denote the set $\{1,2,\ldots,a\}$.
	
	\begin{lemma}\label{n-2}
		Let $G$ be an $s$-connected graph of order $n\ge 7$ with $\Delta(G)=n-2$ and $e(G)\ge {{n-s-1}\choose{2}}+2s$, where $1\leq s\leq 2$. Then either $G$ has a HIST or $G\cong L_n$ and $s=1$. 	
	\end{lemma}
	
	\noindent\p  Let $u$ be a vertex of $G$ with $d(u)=n-2$ and $N(u)=\{u_i:i\in[n-2]\}$. Then $|V(G)\setminus N[u]|=1$. Denote $V(G)\setminus N[u]=\{w\}$. Then $s\le d(w)\le n-2$ as $G$ is $s$-connected. W.l.o.g., we assume that $wu_i\in E(G)$ for all $i\in [s]$. If $u_iu_j\in E(G)$ for some $i\in [s]$, $j\in [n-2]\setminus\{i\}$, then $G$ has a HIST with edge set
$$
E(u,N(u)\setminus \{u_j\})\cup\{wu_i,u_iu_j\}.
$$
Therefore we can assume that $d(u_i)=2$ for all $i\in [s]$. Let $f(x):={{n-x-1}\choose 2}+2x$. Then
	\eqn{}
	{ {{n-s-1}\choose{2}}+2s\le e(G)&=&e(G-N[w])+e(N[w]\cup\{u\})\nonumber\\
		&\le& {{n-d(w)-1}\choose {2}}+2d(w)\le \max\{f(s),f(n-2) \}\nonumber\\
		&=&\max\left\{ {{n-s-1}\choose {2}}+2s, 2(n-2) \right\},\nonumber\\
		&=&{{n-s-1}\choose {2}}+2s,~~~~(n\ge 7)\nonumber
	}
	which implies that $e(G)={{n-s-1}\choose 2}+2s$ and all the above inequalities must be equal, and then, $d(w)=s$ and $G-N[w]$ is a clique of order $n-s-1$. If $s=2$, then $G-u$ is disconnected, a contradiction. Therefore $d(w)=s=1$, and $G-\{w,u_1\}$ is a clique of order $n-2$. Thus $G\cong L_n$. \q

	\begin{lemma}\label{n-t}
		Let $G$ be a connected graph of order $n\ge 8$ with $\Delta(G)=n-t$ and $e(G)\ge\frac{n^2-(t+1)n+2t}{2}$. Then $G$ has a HIST.
	\end{lemma}
	
	\noindent\p First we will show that  $d(x)+d(y)\ge n-1$ for any nonadjacent vertices $x$ and $y$ in $G$. Otherwise, there exist two nonadjacent vertices $x'$ and $y'$ in $G$ such that $d(x')+d(y')\le n-2$. Then
	\eqn{}
	{n^2-(t+1)n+2t\le	2e(G)&=&\sum\limits_{v\in V(G)}d(v)=\sum\limits_{v\in V(G)\setminus\{x',y'\}}d(v)+(d(x')+d(y'))\nonumber\\
		&\le&(n-2)\cdot\Delta(G)+(n-2)=n^2-(t+1)n+2t-2,\nonumber}
	a contradiction. Thus by Theorem~\ref{Thm1.1}, $G$ has a HIST.\q	
	
	\begin{lemma} \label{le2.5}
		Suppose that $G$ is a connected graph of order $n\ge 13$ with $e(G)\ge{{n-3}\choose2}+4$. Let $U\subset V(G)$ with $|U|=3$. If $e(G[U])+ e(U,V(G)\setminus U)\le8$, then $G-(U\cup\{x'\})$ has a HIST for any vertex $x'\in V(G)\setminus U$.
	\end{lemma}
	
	\noindent\p Let $x'\in V(G)\setminus U$, and let $G'=G-(U\cup\{x'\})$. Then $|G'|=n-4\ge9$. We first claim that $\delta(G')\ge|G'|-5$, for otherwise, if there exists a vertex $x^*\in V(G')$ such that $d_{G'}(x^*)\le |G'|-6=n-10$, then
$$ e(G')=e(G'-x^*)+d_{G'}(x^*)\le{{n-5}\choose2}+(n-10	)={{n-4}\choose2}-5.
$$
Therefore,
\eqn{}
	{e(G)&\le&e(G')+e(G[U])+ e(U,V(G)\setminus U)+d_{G'}(x')\nonumber\\
		&\le&{{n-4}\choose2}-5+8+(n-4)=\frac{n^2-7n+18}{2}<{{n-3}\choose2}+4,\nonumber}
	a contradiction. So we have $\delta(G')\ge\frac{|G'|-1}{2}$ as $|G'|\ge9$. Thus, for any two nonadjacent vertices $x,y\in V(G')$, we have $d_{G'}(x)+d_{G'}(y)\ge 2\delta(G')\ge |G'|-1$. Hence, by Theorem~\ref{Thm1.1}, $G'$ has a HIST.\q	

	\begin{lemma}\label{le2.6}
		Suppose that $G$ is a 2-connected graph of order $n\ge13$ with $\Delta(G)=n-t$ and $e(G)\ge{{n-3}\choose2}+4$, where $3\le t\le4$. Let $x\in V(G)$ with $N(x)=\{x_1,x_2\}$. If $3\le\max\{d(x_1),d(x_2)\}\le t$, then $G$ has a HIST.
	\end{lemma}
	
	\noindent\p Let $U=V(G)\setminus N[x]$, and $U_i=N_U(x_i)$ for each $1\le i\le2$. W.l.o.g., we assume that $|U_1|\le|U_2|$. Then $|U_1|\le|U_2|\le 3$ as $d(x_i)\le t\le4$ for $i=1,2$. Set $U'=U\setminus(U_1\cup U_2)$. Then $|U'|=n-3-(|U_1|+|U_2|)\ge n-9\ge4$. Note that $E(u,N[x])=\emptyset$ for each vertex $u\in U'$, it follows that $d(u)\le n-4$.  We consider the following two cases.
	
	\vspace{2mm}
	\noindent{\bf Case 1.} $x_1x_2\in E(G)$.
	\vspace{2mm}
	
	In this case,  $1\le|U_2|\le2$ and $|U'|=n-3-(|U_1|+|U_2|)\ge n-7$. We first claim $d(u')\ge3$ for each $u'\in U_2$. Otherwise, there exists a vertex $u^*\in U_2$ such that $d(u^*)\le2$. Then
	\eqn{}
	{2e(G)=\sum\limits_{v\in V(G)}d(v)&=&\sum\limits_{v\in U'}d(v)+\sum\limits_{v\in U_1\cup U_2}d(v)+\sum\limits_{v\in N[x]}d(v)\nonumber\\
		&\le&|U'|\cdot(n-4)+(|U_1|+|U_2|-1)\cdot\Delta(G)+d(u^*)+10\nonumber\\
		&\le&|U'|\cdot(n-4)+(n-4-|U'|)\cdot(n-3)+12\nonumber\\
		&=&n^2-7n+24-|U'|\le n^2-7n+24-(n-7)<n^2-7n+20,\nonumber}
	a contradiction.
	
	Let $x_2'\in U_2$. Then $d(x_2')\ge3$. Since $xx_2'\notin E(G)$, $N(x_2')\setminus\{x,x_1,x_2\}\neq\emptyset$, say $x_2^*\in N(x_2')\setminus\{x,x_1,x_2\}$. Note that
	$e(G[N[x]])+e(N(x),U_1\cup U_2)\le3+4=7$. Then by Lemma~\ref{le2.5}, $G-(N[x]\cup\{x_2^*\})$ has a HIST $T_1$, and thus $G$ has a HIST with edge set $E(T_1)\cup\{xx_2,x_1x_2,x_2x_2',x_2'x_2^*\}$ (see Figure 3(a)).

\begin{center}	
	\begin{tikzpicture}[scale=0.65]
		\tikzstyle{every node}=[font=\large,scale=0.77]
		
		
		\filldraw (5,5) circle(0.7ex);
		\filldraw (3.5,3.5) circle(0.7ex);
		\filldraw (6.5,3.5) circle(0.7ex);
		\filldraw (6.5,1.5) circle(0.7ex);	
		\filldraw (5,2.5) circle(0.7ex);
		
		\draw (5,5) node[above=2pt]{$x$};
		\draw (3.5,3.5) node[left=2pt]{$x_1$};
		\draw (6.5,3.5) node[right=2pt]{$x_2$};
		\draw (6.6,1.5) node[below=2pt]{$x_2'$};
		\draw (5,2.7) node[left=2pt]{$x_2^*$};
		
		\draw[dashed,line width=0.5mm] (2.5,-0.8) rectangle (7.5,2.1);
		
		\draw (7.5,0) node[left=0pt]{$G-(N[x]\cup\{x_2^*\})$};	
		
		\draw[thick,dashed] (5,5)--(3.5,3.5);
		\draw[very thick] (5,5)--(6.5,3.5);
		\draw[very thick] (3.5,3.5)--(6.5,3.5);
		\draw[very thick] (6.5,3.5)--(6.5,1.5);
		\draw[very thick] (5,2.5)--(6.5,1.5);
		
		\draw (5,-1.3) node[below=0pt]{(a)};
		
		
		\filldraw (12.8,5) circle(0.7ex);
		\filldraw (11.3,3.5) circle(0.7ex);
		\filldraw (14.3,3.5) circle(0.7ex);
		\filldraw (12.8,1.6) circle(0.7ex);	
		\filldraw (15.3,1.6) circle(0.7ex);	
		
		\draw[thick,dashed] (12.8,5)--(11.3,3.5);
		\draw[very thick] (14.3,3.5)--(12.8,5);
		\draw[very thick] (11.3,3.5)--(12.8,1.6);
		\draw[very thick] (14.3,3.5)--(12.8,1.6);
		\draw[very thick] (14.3,3.5)--(15.3,1.6);			
		
		\draw[dashed,line width=0.5mm] (10,-0.8) rectangle (14.9,2.1);
		
		\draw (12.8,5) node[above=2pt]{$x$};
		\draw (11.3,3.5) node[left=2pt]{$x_1$};
		\draw (14.3,3.5) node[right=2pt]{$x_2$};
		\draw (12.8,1.7) node[below=2pt]{$x'$};
		\draw (15.5,1.7) node[below=2pt]{$x_2'$};
		
		\draw (14.9,0) node[left=0pt]{$G-(N[x]\cup\{x_2'\})$};
		
		\draw (12.8,-1.3) node[below=0pt]{(b)};
		
		
		\filldraw (20.6,5) circle(0.7ex);
		\filldraw (19.1,3.5) circle(0.7ex);
		\filldraw (22.1,3.5) circle(0.7ex);
		
		\filldraw (19.9,2.2) circle(0.7ex);	
		\filldraw (21.3,2.2) circle(0.7ex);
		\filldraw (22.9,2.2) circle(0.7ex);
		
		\filldraw (19.1,1) circle(0.7ex);
		\filldraw (22.1,1) circle(0.7ex);
		
		\draw[thick,dashed] (20.6,5)--(19.1,3.5);
		\draw[very thick] (20.6,5)--(22.1,3.5);
		\draw[very thick] (19.1,3.5)--(19.9,2.2);
		\draw[very thick] (22.1,3.5)--(21.3,2.2);
		\draw[very thick] (22.1,3.5)--(22.9,2.2);
		
		\draw[very thick] (19.9,2.2)--(21.3,2.2);
		\draw[very thick] (19.9,2.2)--(19.1,1);
		\draw[very thick] (21.3,2.2)--(22.1,1);
		
		\draw[dashed,line width=0.5mm] (17.6,-1.1) rectangle (23.6,0.4);
		
		\draw (20.6,5) node[above=2pt]{$x$};
		\draw (19.1,3.5) node[left=2pt]{$x_1$};
		\draw (22.1,3.5) node[right=2pt]{$x_2$};
		
		\draw (19.9,2.2) node[left=2pt]{$u_1$};
		\draw (21.3,2.2) node[right=2pt]{$u_2$};
		\draw (22.9,2.2) node[right=2pt]{$u_2'$};
		
		\draw (19.1,1.1) node[left=2pt]{$u^*$};
		\draw (22.1,1.1) node[right=2pt]{$u_2^*$};
		
		\draw (20.6,-0.35) node{$G-U^*$};
		
		\draw (20.6,-1.3) node[below=0pt]{(c)};
		
		
		\draw (12.8,-3.1) node[below=0pt]
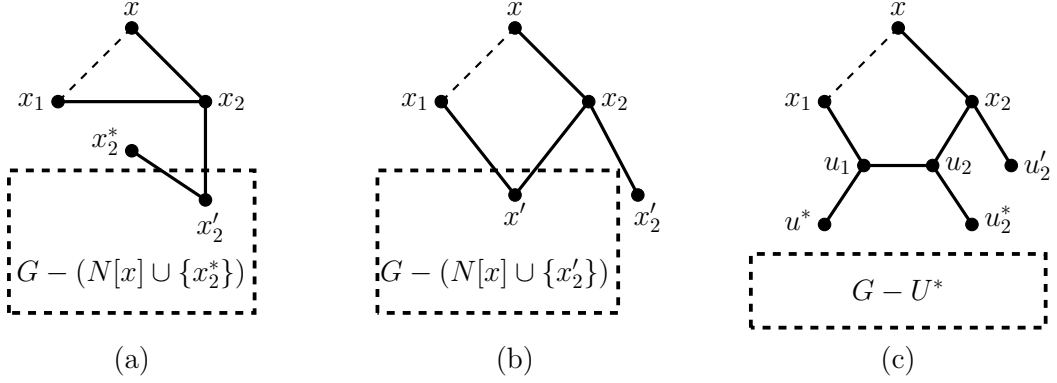
		{\large Figure~3.~~Three HITs of $G$ (strong lines denote edges of HIT)};
		
	\end{tikzpicture}
\end{center}
	
	\vspace{2mm}
	\noindent{\bf Case 2.} $x_1x_2\notin E(G)$.
	\vspace{2mm}
	
	In this case, we have $1\le|U_1|\le3$ and $2\le|U_2|\le3$. If $U_1\cap U_2\neq\emptyset$, say  $x'\in U_1\cap U_2$, then we can say $x_2'\in U_2\setminus\{x'\}$ as $|U_2|=2$. Note that
	$e(G[N[x]])+e(N(x),U_1\cup U_2)\le2+6=8$. Then by Lemma~\ref{le2.5}, $G-(N[x]\cup\{x_2'\})$ has a HIST $T_2$, and thus $G$ has a HIST with edge set $E(T_2)\cup\{xx_2,x'x_1,x'x_2,x_2x_2'\}$ (see Figure 3(b)). So in the following we assume that $U_1\cap U_2=\emptyset$. Then $|U'|\ge n-2t-1$ for $3\le t\le4$ as $|U_1|+|U_2|\le 2(t-1)$.	 Recall that $d(x_1)+d(x_2)\le 2t$.  We will give the following three facts under the assumption of Case 2.
	
	\vspace{2mm}
	{\bf Fact 1.} $E(U_1,U_2)\neq\emptyset$.
	\vspace{2mm}
	
	{\bf Proof of Fact 1.} If $E(U_1,U_2)=\emptyset$, then
	$$e(G-N[x])\le{{n-3}\choose2}-|U_1|\cdot|U_2|,$$ implying that
	\eqn{}
	{e(G)&=&e(G-N[x])+(|U_1|+1)+(|U_2|+1)\le{{n-3}\choose2}-|U_1|\cdot|U_2|+|U_1|+|U_2|+2\nonumber\\
		&=&{{n-3}\choose2}-(|U_1|-1)\cdot(|U_2|-1)+3\le{{n-3}\choose2}+3<{{n-3}\choose2}+4,\nonumber}
	a contradiction. \q
	
	\medskip
	
	By Fact 1, we can let $u_1\in U_1$ and $u_2\in U_2$ with $u_1u_2\in E(G)$. Note that $d(v)\le n-4$ for all $v\in U'$.

		\vspace{2mm}
	{\bf Fact 2.} $d(u_2)\ge5$.
	\vspace{2mm}
	
	{\bf Proof of Fact 2.} If $d(u_2)\le4$, then \eqn{}
	{\sum\limits_{v\in V(G)}d(v)&=&\sum\limits_{v\in U'}d(v)+\sum\limits_{v\in (U_1\cup U_2)\setminus\{u_2\}}d(v)+\sum\limits_{v\in N[x]}d(v)+d(u_2)\nonumber\\
		&\le&|U'|\cdot(n-4)+(|U_1|+|U_2|-1)\cdot\Delta(G)+d(x)+d(x_1)+d(x_2)+4\nonumber\\
		&=&|U'|\cdot(n-4)+(n-4-|U'|)\cdot\Delta(G)+6+d(x_1)+d(x_2)\nonumber\\
		&=&|U'|\cdot(n-4-\Delta(G))+(n-4)\cdot\Delta(G)+6+d(x_1)+d(x_2)\nonumber\\
		&\le&(n-2t-1)(t-4)+(n-4)(n-t)+6+2t\nonumber\\
		&=&n^2-8n+10-2t^2+13t\le n^2-8n+31<n^2-7n+20,\nonumber}
	a contradiction. \q

	\medskip
	
	\vspace{2mm}
	{\bf Fact 3.} {\em There exists some vertex $u^*\in U'$ such that $d(u^*)=n-4$, that is, $u^*$ is adjacent to all vertices of $U\setminus\{u^*\}$.}
	\vspace{2mm}
	
	{\bf Proof of Fact 3.} Suppose that $d_{U}(z')\le|U|-2\le n-5$ for each vertex $z'\in U'$. Since $E(z',\{x,x_1,x_2\})=\emptyset$, we have $d(z')\le n-5$, and thus
	\eqn{}
	{\sum\limits_{v\in V(G)}d(v)&=&\sum\limits_{z'\in U'}d(z')+\sum\limits_{v\in U_1\cup U_2}d(v)+\sum\limits_{v\in N[x]}d(v)\nonumber\\
		&\le&|U'|\cdot(n-5)+(|U_1|+|U_2|)\cdot\Delta(G)+d(x)+d(x_1)+d(x_2)\nonumber\\
		&\le&|U'|\cdot(n-5)+(n-3-|U'|)\cdot\Delta(G)+2+d(x_1)+d(x_2)\nonumber\\
		&\le&(n-2t-1)(t-5)+(n-3)(n-t)+2+2t\nonumber\\
		&=&n^2-8n+7-2t^2+14t=n^2-8n+31<n^2-7n+20,\nonumber}
	a contradiction. \q

	Since $U_1\cap U_2=\emptyset$ and $|U_2|\ge2$, we have $U_2\setminus\{u_2\}\neq\emptyset$, and then we can let $u_2'\in U_2\setminus\{u_2\}$. Note that $x_1u_2\notin E(G)$ and $xu_2\notin E(G)$, then by Fact 2, $N(u_2)\setminus\{x,x_1,x_2,u_1,u_2',u^*\}\neq\emptyset$. Say $u_2^*\in N(u_2)\setminus\{x,x_1,x_2,u_1,u_2',u^*\}$. Let $U^*=\{x,x_1,x_2,u_1,u_2,u_2',u_2^*,u^*\}$. Then $|V(G)\setminus U^*|=n-8\ge5$. Clearly, $G[U^*]$ has a HIST $T$ with edge set $\{xx_2,x_2u_2,x_2u_2',x_1u_1,u_1u_2,u_2u_2^*,$
$u_1u^*\}$ (see Figure 3(c)). By Fact 3, $u^*$ is adjacent to all vertices of $V(G)\setminus U^*$. Therefore, $G$ has a HIST with edge set $E(T)\cup E(u^*,V(G)\setminus U^*)$.\q
	\medskip
	
	\begin{lemma}\label{le2.7}
		Let $G$ be a connected graph of order $n\ge12$ with $\Delta(G)=n-t$ and $e(G)\ge{{n-3}\choose 2}+4$, and let $z\in V(G)$.
		\begin{enumerate}
			\item[(i)] If $d(z)\ge3$ and $t\in\{3,4\}$, then there exists a neighbor $z'\in N(z)$ such that $d(z')\ge t+1$.
			
			\item[(ii)] If $d(z)\ge2$ and $t=5$, then there exists a neighbor $z'\in N(z)$ such that $d(z')\ge 5$.
		\end{enumerate}
	\end{lemma}
	
	\noindent\p (i) Suppose that $d(z'')\le t$ for each vertex $z''\in N(z)$. Let $g(x):={{n-x-1}\choose 2}+tx$. Then
	\eqn{}{ e(G)&=&e(G-N[z])+e(N[z])+e(N(z),V(G)\setminus N[z])\nonumber\\
		&\le&{n-d(z)-1\choose2}+t\cdot d(z)\le\max\{g(3),g(n-t)\}\nonumber\\
		&=&\frac{n^2-9n+20+6t}{2}\le\frac{n^2-9n+44}{2}<{{n-3}\choose2}+4,\nonumber}
	a contradiction. Hence (i) holds.
	
	(ii) Suppose that $d(z^*)\le 4$ for each vertex $z^*\in N(z)$. Note that $\Delta(G)=n-5$ and $d(z)\ge2$. Then
	\eqn{}
	{
		2e(G)=\sum\limits_{v\in V(G)}d(v)&=&d(z)+\sum\limits_{z^*\in N(z)}d(z^*)+\sum\limits_{v\in V(G)\setminus N[z]}d(v)\nonumber\\
		&\le& d(z)+4\cdot d(z)+(n-d(z)-1)\cdot\Delta(G)\nonumber\\
		&=&n^2-6n+5+2(10-n)\le n^2-6n+5+20-2n\nonumber\\
&<&n^2-7n+20,\nonumber
	}
	a contradiction. Hence (ii) holds. \q \medskip
	
	\subsection{Some special HITs in $G$}
	
	In this subsection, we first introduce an important subclass of trees in order to formulate our results.

	\begin{center}	
		\begin{tikzpicture}[scale=0.59]
			\tikzstyle{every node}=[font=\large,scale=0.75]
			\filldraw (5,4.9) circle(0.7ex);
			\filldraw (3,3) circle(0.7ex);
			\filldraw (4.5,3) circle(0.7ex);
			\filldraw (7,3) circle(0.7ex);
			
			\draw (5,4.9) node[above=2pt]{$v_1$};
			
			\draw[very thick] (5,4.9)--(3,3);
			\draw[very thick] (5,4.9)--(4.5,3);
			\draw[very thick] (5,4.9)--(7,3);
			\node [] at (5,2.2) { $\underbrace{~~~~~~~~~~~~~~~~~~~~~~~}_ {r_1}$};

			\filldraw (5.4,3.1) circle(0.3ex);
			\filldraw (6.2,3.1) circle(0.3ex);
			\filldraw (5.8,3.1) circle(0.3ex);
			\draw (5,1) node[below=0pt]{\large (a)~~$S_{r_1}(v_1)$};
			\filldraw (10.5,4.9) circle(0.7ex);
			\filldraw (13.5,4.9) circle(0.7ex);
			\filldraw (9.5,3) circle(0.7ex);
			\filldraw (11.5,3) circle(0.7ex);
			\filldraw (12.5,3) circle(0.7ex);
			\filldraw (14.5,3) circle(0.7ex);
			
			\draw (10.5,4.9) node[above=2pt]{$v_1$};
			\draw (13.5,4.9) node[above=2pt]{$v_2$};
			
			\filldraw (10.5,3.1) circle(0.3ex);
			\filldraw (10.9,3.1) circle(0.3ex);
			\filldraw (10.1,3.1) circle(0.3ex);
			
			\filldraw (13.5,3.1) circle(0.3ex);
			\filldraw (13.9,3.1) circle(0.3ex);
			\filldraw (13.1,3.1) circle(0.3ex);
			
			\filldraw (5.4,3.1) circle(0.3ex);
			\filldraw (6.2,3.1) circle(0.3ex);
			\filldraw (5.8,3.1) circle(0.3ex);
			\draw[very thick] (10.5,4.9)--(13.5,4.9);
			\draw[very thick] (10.5,4.9)--(9.5,3);
			\draw[very thick] (10.5,4.9)--(11.5,3);
			\draw[very thick] (13.5,4.9)--(12.5,3);
			\draw[very thick] (13.5,4.9)--(14.5,3);
			\node [] at (10.5,2.2) {$\underbrace{~~~~~~~~~~~~}_ {r_1}$};
			\node [] at (13.5,2.2) {$\underbrace{~~~~~~~~~~~~}_ {r_2}$};
			\draw (12,1) node[below=0pt]{\large (b)~~$S_{r_1,r_2}(v_1,v_2)$};
			\filldraw (18,4.9) circle(0.7ex);
			\filldraw (20,4.9) circle(0.7ex);
			\filldraw (22,4.9) circle(0.7ex);		
			\filldraw (16.7,3) circle(0.7ex);		
			\filldraw (18.3,3) circle(0.7ex);		
			\filldraw (19.2,3) circle(0.7ex);		
			\filldraw (20.8,3) circle(0.7ex);
			\filldraw (21.7,3) circle(0.7ex);
			\filldraw (23.3,3) circle(0.7ex);
			
			\draw (18,4.9) node[above=2pt]{$v_1$};
			\draw (20,4.9) node[above=2pt]{$v_2$};
			\draw (22,4.9) node[above=2pt]{$v_3$};
			
			\filldraw (17.5,3.1) circle(0.3ex);
			\filldraw (17.2,3.1) circle(0.3ex);
			\filldraw (17.8,3.1) circle(0.3ex);
			
			\filldraw (20,3.1) circle(0.3ex);
			\filldraw (20.3,3.1) circle(0.3ex);
			\filldraw (19.7,3.1) circle(0.3ex);
			
			\filldraw (22.5,3.1) circle(0.3ex);
			\filldraw (22.8,3.1) circle(0.3ex);
			\filldraw (22.2,3.1) circle(0.3ex);
			
			\draw[very thick] (18,4.9)--(20,4.9);
			\draw[very thick] (20,4.9)--(22,4.9);
			\draw[very thick] (18,4.9)--(18.3,3);
			\draw[very thick] (18,4.9)--(16.7,3);
			\draw[very thick] (20,4.9)--(19.2,3);
			\draw[very thick] (20,4.9)--(20.8,3);
			\draw[very thick] (22,4.9)--(21.7,3);
			\draw[very thick] (22,4.9)--(23.3,3);
			\node [] at (17.5,2.2) {$\underbrace{~~~~~~~~~~}_ {r_1}$};
			\node [] at (20,2.2) {$\underbrace{~~~~~~~~~~}_ {r_2}$};
			\node [] at (22.5,2.2) {$\underbrace{~~~~~~~~~~}_ {r_3}$};
			
			\draw (20,1) node[below=0pt]{\large (c)~~$S_{r_1,r_2,r_3}(v_1,v_2,v_3)$};
		
			\draw (12,-1.1) node[below=0pt]{\large Figure~4.~~Some special trees};
\end{tikzpicture}
\end{center}

Let $S_{r_1, r_2,\ldots, r_t}(v_1,v_2,\ldots,v_t)$ be the tree obtained from a path $P_t:=v_1v_2\cdots v_t$ by adding $r_i$ leaves to the vertex $v_i$ for each $i\in[t]$. Thus $S_{r_1}(v_1)$ is a star of order $r_1+1$ with center $v_1$ (see Figure 4(a)), $S_{r_1,r_2}(v_1,v_2)$ a double star of order $r_1+r_2+2$ with two stems $v_1,v_2$ when $r_i\ge 2$ for $i\in [2]$ (see Figure 4(b)), and $S_{r_1,r_2,r_3}(v_1,v_2,v_3)$ a tree of order $r_1+r_2+r_3+3$ with three stems $v_1,v_2,v_3$ when $r_1,r_3\ge 2$ and $r_2\ge 1$ (see Figure 4(c)). In particular, we define $T^1:=S_{2,1,r_3}(v_1,v_2,v_3)$ with $r_3\ge3$ and  $T^2:=S_{2,r_2,2}(v_1,v_2,v_3)$ with $r_2\ge3$.

	\begin{lemma}\label{le2.8}
		Suppose that $G$ is a connected graph of order $n\ge9$ with $n-5\le\Delta(G)\le n-3$ and $e(G)\ge{{n-3}\choose2}+4$. Let $u\in V(G)$ with $d(u)=\Delta(G)$. Denote $T:=S_{\Delta(G)}(u)$ and $W:=V(G)\setminus V(T)$. If $1\le|N_T(W)|\le\Delta(G)$, then there exists a vertex $u'\in N_{T}(W)$ such that $N_{T-u}(u')\neq\emptyset$.
	\end{lemma}
	\noindent\p
	Since $|T|=\Delta(G)+1$, we have $n-4\le|T|\le n-2$ and hence $2\le|W|\le4$. Suppose that $N_{T-u}(u'')=\emptyset$ for each vertex $u''\in N_{T}(W)$, that is, $d(u'')=2$.  Let
$$
f(x,y):={{n-x-y}\choose 2}+{x\choose2}+2y.
$$ Thus
	\eqn{}
	{   e(G)&=&e(G-N_{T}[W])+e(N_T[W]\cup\{u\})=e(G-N_{T}[W])+e(G[W])+\sum\limits_{u''\in N_{T}(W)}d(u'')\nonumber\\
		&\le&{n-|W|-|N_T(W)|\choose2}+{|W|\choose2}+2|N_T(W)|\le\max\{f(|W|,1),f(|W|,n-3)\}\nonumber\\
		&\le&\max\{f(2,1),f(4,1),f(2,n-3),f(4,n-3)\}=\frac{n^2-7n+18}{2}<\frac{n^2-7n+20}{2},\nonumber
	}
	a contradiction.\q\medskip

	
	\begin{lemma}\label{le2.10}
		Let $G$ be a connected graph of order $n\ge9$, and let $u,v\in V(G)$. Suppose that $G$ has a HIT $T\in\{S_{n-5,2}(u,v),S_{n-6,2}(u,v),S_{n-6,3}(u,v)\}$ which is not $(1,2)$-extendable and $w\in V(G)\setminus V(T)$. If there exists a vertex $w'\in V(T)$ such that $ww'\in E(G)$ and $d(w')\ge4$, then $G$ has a HIT $T'$ with $V(T')=V(T)\cup\{w\}$. In particular, if $T\cong S_{n-6,3}(u,v)$, then $T'$ is a HIST of $G$; and if $T\in\{S_{n-5,2}(u,v),S_{n-6,2}(u,v)\}$, then $T'\in\{T^1, T^2, S_{a,3}(u,w')\}$.
	\end{lemma}
	
\noindent\p Set $W=V(G)\setminus V(T)$. Then $1\le |W|\le 2$. Since $T$ is not a $(1,2)$-extendable HIT, $w'$ is a leaf of $T$ and $d_W(w')=1$. Clearly, $d_T(u)\ge n-5\ge4$, so $u$ is a nice stem of $T$. Denote $A:=N_{T}(u)\setminus\{v\}$ and $B:=N_{T}(v)\setminus\{u\}$. Then $|A|\ge3$ and $2\le|B|\le3$.

If $|B|=3$, then $|A|=n-6$ and $|W|=1$. Note that $d_T(v)=|B|+1=4$, that is, $v$ is a nice stem of $T$. Since $d(w')\ge 4$, we have $N(w')\setminus\{u,v,w\}\neq\emptyset$, which implies that $N_T(w')\cap (A\cup B)\neq\emptyset$. Then by Lemma~\ref{le2.2}, $T_{+w}$ is a HIST of $G$.

Now, we consider the case that  $|B|=2$.	
	If $N_{T}(w')\cap A\neq\emptyset$, then by Lemma~\ref{le2.2}, $T_{+w}$ is a HIT of $G$ with $V(T_{+w})=V(T)\cup\{w\}$, and we can check that $T_{+w}\cong T^1$ or $T_{+w}\cong T^2$. So we assume that $N_{T}(w')\cap A=\emptyset$. If $w'\in A$ and $vw'\in E(G)$, then $T_1:=T+\{ww',vw'\}-\{uv\}$ is a HIT of $G$ with $V(T_1)=V(T)\cup\{w\}$ and $T_1\cong T^1$; and if $w'\in A$ and $vw'\notin E(G)$, let $N_T(v)\setminus\{u\}=\{v_1,v_2\}$, then $N(w')=\{u,w,v_1,v_2\}$, and thus $T_2:=T+\{ww',v_1w',v_2w'\}-\{vv_1,vv_2\}$ is a HIT of $G$ with $V(T_2)=V(T)\cup\{w\}$ and $T_2\cong S_{a,3}(u,w')$.  \q

	\begin{lemma}\label{le2.11}
		Let $G$ be a connected graph of order $n\ge10$, and let $u,x,y\in V(G)$. Suppose that $G$ has a HIT $T\in\{S_{2,1,n-7}(x,y,u), S_{2,n-8,2}(x,u,y)\}$ and  $V(G)\setminus V(T)=\{w\}$. If there exists a vertex $w'\in N_T(w)$ such that $d(w')\ge5$, then $G$ has a HIST.
	\end{lemma}
	\noindent\p Note that $T$ contains three stems $u$, $x$ and $y$. Clearly, we have $d_T(u)=n-6\ge4$, that is, $u$ is a nice stem of $T$. By Lemma~\ref{le2.1}, we may assume that $T$ is not a $(1,2)$-extendable HIT. Let $w'\in N_T(w)$ such that $d(w')\ge5$. Then $w'\notin\{u,x,y\}$.
	Set $N_T(x)\setminus\{u\}=\{x_1,x_2\}$, $N_T(y)\setminus\{u\}=\{y_1,y_2\}$ and $U=\{x_1,x_2,y_1,y_2\}$.

	If $w'\in N_T(u)$ and $N_{T}(w')\cap N_{T}(u)\neq\emptyset$, then by Lemma~\ref{le2.2}, $T_{+w}$ is a HIST of $G$; and if $w'\in N_T(u)$ and $N_T(w')\cap N_T(u)=\emptyset$, then $d_U(w')=d(w')-2\ge3$. It follows that $x_1w',x_2w'\in E(G)$, or $y_1w',y_2w'\in E(G)$. W.l.o.g., assume that $x_1w',x_2w'\in E(G)$. Then $G$ has a HIST with edge set $(E(T)\setminus\{xx_1,xx_2\})\cup\{x_1w',x_2w',ww'\}$. So in the following, we assume that $w'\notin N_T(u)$, that is, $w'\in U$. We consider the following two cases.
	
	\vspace{2mm}
	
	\noindent{\bf Case 1.} $T\cong S_{2,1,n-7}(x,y,u)$.
	
	\vspace{2mm}
	
	In this case, we may assume that $y_2=x$. If $N_T(w')\cap (N_T(u)\setminus\{y\})\neq\emptyset$, then by Lemma~\ref{le2.2}, $T_{+w}$ is a HIST of $G$. So we assume that $N_T(w')\cap (N_T(u)\setminus\{y\})=\emptyset$. Then $N(w')\subset U\cup\{w,y\}$. Since $d(w')\ge5$, we have $N(w')=\{x,y,w,x_1,x_2,y_1\}\setminus\{w'\}$. Therefore,
	\[T'=
	\left\{
	\begin{array}{ll}
		T+\{w'x_{3-i},w'y,ww'\}-\{xx_1,xx_2\}, & \mbox{if} ~ w'=x_i;\\
		T+\{x_1w',x_2w',ww'\}-\{xx_1,xx_2\}, & \mbox{if} ~ w'=y_1.
	\end{array}
	\right.
	\]
	is a HIST of $G$, where $i\in[2]$.
	
	\vspace{2mm}
	
	\noindent{\bf Case 2.} $T\cong S_{2,n-8,2}(x,u,y)$.
	
	\vspace{2mm}
	
	If $N_T(w')\cap (N_T(u)\setminus\{x,y\})\neq\emptyset$, then by Lemma~\ref{le2.2}, $T_{+w}$ is a HIST of $G$. So we assume that $N_T(w')\cap (N_T(u)\setminus\{x,y\})=\emptyset$. Then $N_T(w')\subset U\cup\{x,y\}$.
	Recall that $w'\in U$. By symmetry, we assume that $w'\in\{x_1,x_2\}$. Say $w'=x_1$. If $w'y\in E(G)$, then by Lemma~\ref{le2.2}, $T_{+w}$ is a HIST of $G$. So we assume $w'y\notin E(G)$. Then $w'y_1\in E(G)$ and $w'y_2\in E(G)$ as $d(w')\ge5$. Therefore, $T+\{ww',w'y_1,w'y_2\}-\{yy_1,yy_2\}$ is a HIST of $G$. \q\medskip
	
	\section{Proof of Theorem~\ref{Thm1.4}}	
	
Since $L_n$ is connected, HIST-free and $e(L_n)=\binom{n-2}{2}+2$, we have $\operatorname{ex}^{\mathrm{HIST}}_1(n)\ge \binom{n-2}{2}+2$. It remains to prove the upper bound and the uniqueness of the extremal graph. Thus let $G$ be a connected graph of order $n\ge9$ with $\Delta(G)=\Delta$ and $e(G)\ge\binom{n-2}{2}+2$.  We show that either $G$ has a HIST or $G\cong L_n$.

	Let $u$ be a vertex of $G$ with $d(u)=\Delta$ and $N(u)=\{u_i:i\in[\Delta]\}$. If $n=\Delta+1$, then any spanning star centered at $u$ is a HIST. So we assume that $n>\Delta+1$. Since $e(G)\ge{{n-2}\choose{2}}+2$, we have
	$$ \Delta\ge\frac{2e(G)}{n}\ge\frac{n^2-5n+10}{n}>n-5.$$ That is $n-4\le \Delta\le n-2$. If $\Delta= n-2$, then by Lemma \ref{n-2}, either $G$ has a HIST or $G\cong L_n$; and if $\Delta= n-4$, then by Lemma \ref{n-t}, $G$ also has a HIST. Thus, in the following, we assume that $\Delta=n-3\ge 6$. Denote $V(G)\setminus N[u]=\{w_1,w_2\}$.

The star centered at $u$ forms a HIT $T:=S_{n-3}(u)$ with vertex set $N[u]$ and edge set $E(u,N(u))$. If $(w_1,w_2)$ is a good pair of $V(T)\setminus\{u\}$, then by Lemma~2.1, $G$ has a HIST. So we may assume that $(w_1,w_2)$ is not a good pair of $V(T)\setminus\{u\}$.
Then each vertex of $V(T)\setminus\{u\}$ is adjacent to at most one of $w_1$ and $w_2$, and hence $|N_T(\{w_1,w_2\})|\le |V(T)\setminus\{u\}|=\Delta$. Since $G$ is connected, we have $N_T(\{w_1,w_2\})\neq\emptyset$. Therefore, $1\le |N_T(\{w_1,w_2\})|\le \Delta$.

	By Lemma~\ref{le2.8}, there exists a vertex $u^*\in N_T(w_i)$ such that $N_{T-u}(u^*)\neq\emptyset$ for some $i\in[2]$. Note that $u$ is a nice stem of $T$, and hence, by Lemma~\ref{le2.2}, $T_{+w_i}$ is a HIST of $G-w_{3-i}$. It is easy to see that $T_{+w_i}\cong S_{n-5,2}(u,u^*)$. W.l.o.g., we can assume that $G-w_2$ contains a HIST $T':= S_{n-5,2}(u,u^*)$, where $d_{T'}(u)=n-4\ge5$ and $d_{T'}(u^*)=3$. By Lemma~\ref{le2.1}, we may assume that $T'$ is not a $(1,2)$-extendable HIT of $G$. By the construction of $T'$ and the nonexistence of a good pair,
$w_2$ is adjacent to no stem of $T'$.
	
If there exists a vertex $w_2'\in N_{T'}(w_2)$ such that $d(w_2')\ge4$, then by Lemma~\ref{le2.10}, $G$ has HIT $T''$ with vertex set $V(T')\cup\{w_2\}$. Clearly, $T''$ is a HIST of $G$. So we assume that $d(w')\le3$ for each vertex $w'\in N_{T'}(w_2)$.
	Then
	\eqn{}
	{n^2-5n+10&\le&\sum\limits_{v\in V(G)}d(v)=d(w_2)+\sum\limits_{w'\in N(w_2)}d(w')+\sum\limits_{v\in V(G)\setminus N[w_2]}d(v)\nonumber\\
		&\le&d(w_2)+3\cdot d(w_2)+(n-d(w_2)-1)\cdot\Delta\nonumber\\
		&=&n^2-4n+3+d(w_2)(7-n)\nonumber\\
		&\le&n^2-4n+3+7-n=n^2-5n+10,\nonumber}
	which implies that all the above inequalities must be equal, that is, $d(w_2)=1$ and $d(w_2^*)=3$ for $w_2^*\in N(w_2)$. Let $N(w_2^*)\setminus\{w_2\}=\{x,y\}$, and let $W'=V(G)\setminus N[w_2^*]$. Then $|W'|=n-4\ge5$.
	
Next we claim that either $x$ or $y$ is adjacent to all vertices of $W'$. Otherwise, there exist $x',y'\in W'$ such that $xx'\notin E(G)$ and $yy'\notin E(G)$, which implies that $e(G-\{w_2,w_2^*\})\le{{n-2}\choose2}-2$, and hence
	$$e(G)=e(G-\{w_2,w_2^*\})+3\le{{n-2}\choose2}-2+3<{{n-2}\choose 2}+2,$$ a contradiction. Therefore, $G$ has a HIST with edge set $E(x,W')\cup \{w_2w_2^*,xw_2^*,yw_2^*\}$ or $ E(y,W')\cup \{w_2w_2^*,xw_2^*,yw_2^*\}$.
	
	This completes the proof of Theorem~\ref{Thm1.4}.\q

\section{Proof of Theorem~\ref{Thm1.5}}

Since $B_n$ is $2$-connected, HIST-free and $e(B_n)=\binom{n-3}{2}+4$, we have $\operatorname{ex}^{\mathrm{HIST}}_2(n)\ge \binom{n-3}{2}+4$. It remains to prove the upper bound and the uniqueness of the extremal graph. Thus let $G$ be a $2$-connected graph of order $n\ge13$ with $\Delta(G)=\Delta$ and $e(G)\ge\binom{n-3}{2}+4$.
Then
$$
\Delta\ge\frac{2e(G)}{n}\ge\frac{n^2-7n+20}{n}>n-7.
$$
We show that either $G$ has a HIST or $G\cong B_n$.

If $n=\Delta+1$, then any spanning star centered at a vertex of degree $\Delta$ is a HIST. If $\Delta= n-2$, then by Lemma \ref{n-2}, $G$ has a HIST; and if $\Delta= n-6$, then by Lemma \ref{n-t}, $G$ also has a HIST. Thus, in the following, we assume that  $n-5\le\Delta\le n-3$. Let $u$ be a vertex of $G$ with $d(u)=\Delta$ and $N(u)=\{u_i:i\in[\Delta]\}$, and let $W=V(G)\setminus N[u]$. Then $2\le|W|\le4$.
	
	Let $b:=d_W(u_1)=\max\{d_W(u_i):u_i\in N(u)\}$. Then $d_W(u_1)\ge 1$ as $G$ is connected. If $d_W(u_1)=|W|$, then $G$ has a HIST with edge set $E(u,N(u))\cup E(u_1,W)$. So we assume $1\le d_W(u_1)<|W|$. We first obtain the following claims. \medskip
	
	{\bf Claim 1.} {\em If $n-4\le \Delta \le n-3$ and $d_W(u_1)=1$, then $G$ has a HIT $S_{\Delta-2,2}(u,u_j)$ for some $j\in [\Delta]$.}\medskip
	
	{\bf Proof of Claim 1.} Clearly, $G-W$ has a HIST $T:=S_{\Delta}(u)$ with vertex set $N[u]$ and edge set $E(u,N(u))$. Then $N_{T}(W)\neq\emptyset$ as $G$ is connected. Since $d_W(u_i)\le1$ for each $i\in[\Delta]$, we have $|N_{T}(W)|\le|V(T)\setminus\{u\}|=\Delta$. By Lemma~\ref{le2.8}, there exists a vertex $u_j\in N_{T}(W)$ such that $N_{T-u}(u_j)\neq\emptyset$. Let $w\in W$ with $wu_j\in E(G)$.  Note that $u$ is a nice stem of $T$. By Lemma~\ref{le2.2}, $T_{+w}$ is a HIST of $G-(W\setminus\{w\})$. By construction, $T_{+w}\cong S_{\Delta-2,2}(u,u_j)$. \q \medskip
	
	{\bf Claim 2.} {\em If $n-5\le \Delta \le n-4$ and $2\le d_W(u_1)<|W|$, then $G$ has a HIT $S_{\Delta-1,b}(u,u_1)$.}
	
	\medskip
	
	{\bf Proof of Claim 2.} Since $n-5\le \Delta \le n-4$, we have $3\le|W|\le4$. Note that $d(u)=\Delta\ge7$ and $d_W(u_1)\ge2$. Then $G$ has a HIT $T$ with vertex set $N[u]\cup N_W(u_1)$ and edge set $E(u,N(u))\cup E(u_1,N_W(u_1))$. By construction, $T\cong S_{\Delta-1,b}(u,u_1)$.\q\medskip
	
	{\bf Claim 3.} {\em Let $w\in V(G)$ with $N(w)=\{w^1,w^2\}$. If $n-4\le \Delta\le n-3$, then either $\max\{d(w^1),d(w^2)\}\ge3$, or $G\cong B_n$ and $\Delta=n-3$.}\medskip
	
	{\bf Proof of Claim 3.} Suppose that $d(w^1)=d(w^2)=2$ and $G\ncong B_n$. Then $w^1w^2\notin E(G)$ and $N(w^1)\cap N(w^2)=\{w\}$ as $G$ is 2-connected, moreover, $G-\{w,w^1,w^2\}$ is not a clique of order $n-3$. Thus
	$$
	e(G)= e(G-\{w,w^1,w^2\})+4<{{n-3}\choose2}+4,
	$$
	a contradiction.  \q\medskip
	
	Since $1\le d_W(u_1)<|W|$ and $2\le|W|\le4$, we have $1\le|W\setminus N_W(u_1)|\le3$. We consider the following three cases.
	
	\vspace{3mm}
	
	\noindent{\bf Case 1.} $|W\setminus N_W(u_1)|=1$.
	
	\vspace{3mm}
	
	In this case, we have $d_W(u_1)=|W|-1$ and $\Delta=n-|W|-1$.
	
	If $|W|=2$, then $d_W(u_1)=1$ and $\Delta=n-3$, and thus, by Claim 1, $G$ has a HIT $T:=S_{n-5,2}(u,u_j)$ of order $n-1$ for some $j\in[\Delta]$. If $|W|=3$ (or resp. $|W|=4$), then $d_W(u_1)=2$ and $\Delta=n-4$ (or resp. $d_W(u_1)=3$ and $\Delta=n-5$), and then by Claim 2, $G$ has a HIT $T:=S_{\Delta-1,b}(u,u_1)$ of order $n-1$, where $2\le b\le 3$. Let $V(G)\setminus V(T)=\{w\}$. Then $d(w)\ge 2$ as $G$ is 2-connected.
	
By Lemma~\ref{le2.1}, we may assume that $T$ is not a $(1,2)$-extendable HIT of $G$. Since $V(G)\setminus V(T)=\{w\}$, the vertex $w$ is not adjacent to any stem of $T$. If there exists some $w'\in N_T(w)$ such that $d(w')\ge4$, then by Lemma~\ref{le2.10}, $G$ has a HIT $T'$ with $V(T')=V(T)\cup\{w\}$. Clearly, $T'$ is a HIST of $G$. So we assume that all vertices in $N_T(w)$ have degree at most $3$. By Lemma~\ref{le2.7}, we have $n-4\le \Delta\le n-3$ and $d(w)=2$. Let $N(w)=\{w^1,w^2\}$. By Claim 3, we have either $\max\{d(w^1),d(w^2)\}=3$ or $G\cong B_n$. If $\max\{d(w^1),d(w^2)\}=3$, then by Lemma~\ref{le2.6}, $G$ has a HIST.
	
	\vspace{3mm}
	
	\noindent{\bf Case 2.} $|W\setminus N_W(u_1)|=2$.
	
	\vspace{3mm}
	
	In this case, we have $3\le|W|\le4$. Clearly, $d_W(u_1)=|W|-2$ and $n-5\le \Delta\le n-4$.
	
	If $|W|=3$, then $d_W(u_1)=1$ and $\Delta=n-4$, and then by Claim 1, $G$ has a HIT $T:=S_{n-6,2}(u,u_j)$ of order $n-2$ for some $j\in[\Delta]$. If $|W|=4$, then $d_W(u_1)=2$ and $\Delta=n-5$, and by Claim 2, $G$ has a HIT $T:=S_{n-6,2}(u,u_1)$ of order $n-2$. Let $V(G)\setminus V(T)=\{w_1,w_2\}$. If $(w_1,w_2)$ is a good pair of $V(T)$, then by Lemma~\ref{le2.1}, $G$ has a HIST. So we may assume that $T$ is not a $(1,2)$-extendable HIT of $G$. Then $N_T(w_1)\cap N_T(w_2)=\emptyset$, implying that
	\eqn{e1}{
		d(w_1)+d(w_2)\le n.
	}
	We obtain the following claim under the assumption of the above.	
	\vspace{2mm}
	
	{\bf Claim 4.} {\em There exists a vertex $z\in N_{T}(w_1)\cup N_{T}(w_2)$ with $d(z)\ge4$.}
	
	\vspace{2mm}
	{\bf Proof of Claim 4.} Suppose that $d(z')\le3$ for each $z'\in N_{T}(w_1)\cup N_{T}(w_2)$. If $|N_{T}(w_1)\cup N_{T}(w_2)|\ge 3$, let $z_1,z_2,z_3\in N_{T}(w_1)\cup N_{T}(w_2)$, then by (\ref{e1}), we have
	\eqn{}
	{
		\sum\limits_{v\in V(G)}d(v)&=&\sum\limits_{v\in V(G)\setminus\{w_1,w_2,z_1,z_2,z_3\}}d(v)+d(w_1)+d(w_2)+d(z_1)+d(z_2)+d(z_3)\nonumber\\
		&\le&(n-5)\cdot\Delta+n+9\le(n-5)(n-4)+n+9\nonumber\\
		&=&n^2-8n+29<n^2-7n+20,\nonumber
	}
	a contradiction. So $|N_{T}(w_1)\cup N_{T}(w_2)|\le 2$. On the other hand, since $d(w_1)\ge2$, $d(w_2)\ge2$ and $N_{T}(w_1)\cap N_{T}(w_2)=\emptyset$, we have $|N_{T}(w_1)\cup N_{T}(w_2)|\ge2$. Let $z_4,z_5\in N_{T}(w_1)\cup  N_{T}(w_2)$. Then
	\eqn{}
	{
		\sum\limits_{v\in V(G)}d(v)&=&\sum\limits_{v\in V(G)\setminus\{w_1,w_2,z_4,z_5\}}d(v)+d(w_1)+d(w_2)+d(z_4)+d(z_5)\nonumber\\
		&\le&(n-4)\cdot\Delta+4+3+3\le(n-4)(n-4)+4+3+3\nonumber\\
		&=&n^2-8n+26<n^2-7n+20,\nonumber
	}
	a contradiction. \q\medskip
	
	We may assume that there exists a vertex $w_1'\in N_T(w_1)$ such that $d(w_1')\ge4$. Then by Lemma~\ref{le2.10}, $G$ has a HIT $T'$ with $V(T')=V(T)\cup\{w_1\}$, and $T'\cong T^1$, or $T'\cong T^2$, or $T'\cong S_{n-6,3}(u,w_1')$.
	Obviously, $T'$ is a HIST of $G-w_2$. Recall that $n-5\le \Delta\le n-4$ and $d(w_2)\ge 2$.
	
	We first assume that there exists a vertex $w_2'\in N_{T'}(w_2)$ such that $d(w_2')\ge 5$. If $T'\cong T^1$ or $T'\cong T^2$, then by Lemma~\ref{le2.11}, $G$ has a HIST; and if $T'\cong S_{n-6,3}(u,w_1')$, then by Lemma~\ref{le2.10}, $G$ also has a HIST. Next we assume that every vertex in $N_{T'}(w_2)$ has degree at most $4$. By Lemma~\ref{le2.7}, $\Delta=n-4$ and $d(w_2)=2$. Let $N(w_2)=\{w^1,w^2\}$. Then by Claim 3, we have $3\le\max\{d(w^1),d(w^2)\}\le4$, and thus, by Lemma~\ref{le2.6}, $G$ has a HIST.
	
	\vspace{3mm}
	\noindent{\bf Case 3.} $|W\setminus N_W(u_1)|=3$.
	\vspace{3mm}
	
	In this case, $\Delta=n-5$, $|W|=4$ and $d_W(u_1)=1$. If $e(N(u),W)\le3$, then we have
	\eqn{}
	{
		e(G)&=&e(G[N[u]])+e(G[W])+e(N(u),W)\nonumber\\
		&\le&{{n-4}\choose2}+{4\choose2}+3=\frac{n^2-9n+38}{2}<\frac{n^2-7n+20}{2},\nonumber
	}
	a contradiction. So $e(N(u),W)\ge4$. Note that each vertex $u'\in N_{N(u)}(W)$  is nonadjacent to at least one vertex in $N(u)$, for otherwise, $d(u')\ge |N(u)|+1>\Delta$, a contradiction. Hence, $e(G[N[u]])\le {{n-4}\choose2}-2$.
	On the other hand,
	since $d_W(u_i)\le1$ for each $i\in[\Delta]$, we have
	\eqn{e2}
	{
		e(N(u),W)\le|N(u)|=n-5.
	}
	Thus
	\eqn{}
	{
		e(G)&=&e(G[N[u]])+e(G[W])+e(N(u),W)\nonumber\\
		&\le&{{n-4}\choose2}-2+{4\choose2}+(n-5)=\frac{n^2-7n+18}{2}<\frac{n^2-7n+20}{2},\nonumber
	}
	a contradiction.

	This completes the proof of Theorem~\ref{Thm1.5}.\q
	
	
	\section{Further research}

Theorems~\ref{Thm1.4} and~\ref{Thm1.5} determine
$\operatorname{ex}^{\mathrm{HIST}}_k(n)$ for $k=1,2$ and characterize
the corresponding extremal graphs. A natural problem is to determine
$\operatorname{ex}^{\mathrm{HIST}}_k(n)$ for $k\ge 3$.

We now construct a $k$-connected HIST-free graph $D_{n,k}$ of order $n\ge k^2+2k$ (see Figure 5). Let $S^1,\ldots,S^k$ be $k$ disjoint copies of the star $K_{1,k}$, and let $Y\subseteq V(K_{n-k^2-k})$ with $|Y|=k$. For $i\in[k]$, denote by $V_i$ the set of leaves of $S^i$. The graph $D_{n,k}$ is obtained from $S^1,\ldots,S^k$ and $K_{n-k^2-k}$ by adding all edges between every vertex in $\bigcup_{i=1}^k V_i$ and every vertex in $Y$. Then $D_{n,k}$ is $k$-connected, $D_{n,1}\cong L_n$, and $e(D_{n,k})=\binom{n-k^2-k}{2}+k^2(k+1)$.

\begin{proposition} \label{prop5.1}
	For $k\ge3$, the graph $D_{n,k}$ is HIST-free.
	\end{proposition}
	
\noindent\p   Suppose that $T$ is a HIST of $D_{n,k}$. Note that each $V_i$ is a vertex cut of $D_{n,k}$. By the definition of a HIST, for each $i \in [k]$, there exists $x_i \in V_i$ such that $d_T(x_i) \ge 3$. Otherwise, every vertex of $V_i$ is a leaf of $T$. Then $T-V_i$ is connected, contradicting the fact that $V_i$ is a vertex cut. It follows that each $x_i$ has at least two neighbors from $Y$ in $T$. Let $X := \{x_1, \dots, x_k\}$, and let $E_T(X,Y)$ denote the set of edges between $X$ and $Y$ in $T$. Then
	$$
	|E_T(X,Y)| \ge 2k = |X \cup Y|,
	$$
	which implies that $T[X \cup Y]$ contains a cycle, contradicting the fact that $T$ is a tree.\q

	\vskip.2cm
	\begin{center}
		\begin{tikzpicture}[line cap=round, line join=round, >=latex, x=0.65cm, y=0.65cm, line width=1pt]
			\draw (0, 4.2) ellipse (4.2cm and 1.4cm);
			\node at (0, 5) [align=center, font=\large] { $K_{n-k^2-k}$};
			\draw [dashed] (0, 3.75) ellipse (3.3cm and 0.4cm);
			\node at (3.7, 4.5) [right] {$Y$};
			\draw [dashed] (-6, 0) ellipse (1.6cm and 0.7cm);
			\node at (-8, -1) [left] {$V_1$};
			\draw [dashed] (0, 0) ellipse (1.6cm and 0.7cm);
			\node at (2, -1) [right] {$V_2$};
			\draw [dashed] (6, 0) ellipse (1.6cm and 0.7cm);
			\node at (7.9, -1) [right] {$V_k$};
			\draw [fill=black] (-3, 3.75) circle (2pt);
			\draw [fill=black] (0, 3.75)   circle (2pt);
			\draw [fill=black] (1.25,3.75)   circle (0.7pt);
			\draw [fill=black] (1.5, 3.75)   circle (0.7pt);
			\draw [fill=black] (1.75,3.75)   circle (0.7pt);
			\draw [fill=black] (3, 3.75)  circle (2pt);
			\draw (-3, 3.75) -- (-7.5, 0);
			\draw (-3, 3.75) -- (-6.5, 0);
			\draw (-3, 3.75) -- (-4.5, 0);
			\draw (-3, 3.75) -- (-1.5, 0);
			\draw (-3, 3.75) -- (-0.5, 0);
			\draw (-3, 3.75) -- (1.5, 0);
			\draw (-3, 3.75) -- (4.5, 0);
			\draw (-3, 3.75) -- (5.5, 0);
			\draw (-3, 3.75) -- (7.5, 0);
			\draw (0, 3.75) -- (-7.5, 0);
			\draw (0, 3.75) -- (-6.5, 0);
			\draw (0, 3.75) -- (-4.5, 0);
			\draw (0, 3.75) -- (-1.5, 0);
			\draw (0, 3.75) -- (-0.5, 0);
			\draw (0, 3.75) -- (1.5, 0);
			\draw (0, 3.75) -- (4.5, 0);
			\draw (0, 3.75) -- (5.5, 0);
			\draw (0, 3.75) -- (7.5, 0);
			\draw (3, 3.75) -- (-7.5, 0);
			\draw (3, 3.75) -- (-6.5, 0);
			\draw (3, 3.75) -- (-4.5, 0);
			\draw (3, 3.75) -- (-1.5, 0);
			\draw (3, 3.75) -- (-0.5, 0);
			\draw (3, 3.75) -- (1.5, 0);
			\draw (3, 3.75) -- (4.5, 0);
			\draw (3, 3.75) -- (5.5, 0);
			\draw (3, 3.75) -- (7.5, 0);
			\draw [fill=black] (-7.5, 0) circle (2pt);
			\draw [fill=black] (-6.5, 0) circle (2pt);
			\draw [fill=black] (-4.5, 0) circle (2pt);
			\draw [fill=black] (-6, -3) circle (2pt);
			\draw [fill=black] (-5.5, 0) circle (0.7pt);
			\draw [fill=black] (-5.75, 0) circle (0.7pt);
			\draw [fill=black] (-5.25, 0) circle (0.7pt);
			\draw (-7.5, 0) -- (-6, -3);
			\draw (-6.5, 0) -- (-6, -3);
			\draw (-4.5, 0) -- (-6, -3);
			\draw [fill=black] (-1.5, 0) circle (2pt);
			\draw [fill=black] (-0.5, 0) circle (2pt);
			\draw [fill=black] (1.5, 0) circle (2pt);
			\draw [fill=black] (0, -3) circle (2pt);
			\draw [fill=black] (0.25, 0) circle (0.7pt);
			\draw [fill=black] (0.5, 0) circle (0.7pt);
			\draw [fill=black] (0.75, 0) circle (0.7pt);
			\draw (-1.5, 0) -- (0, -3);
			\draw (-0.5, 0)  -- (0, -3);
			\draw (1.5, 0)  -- (0, -3);
			\draw [fill=black] (4.5, 0) circle (2pt);
			\draw [fill=black] (5.5, 0) circle (2pt);
			\draw [fill=black] (7.5, 0) circle (2pt);
			\draw [fill=black] (6, -3) circle (2pt);
			\draw [fill=black] (6.5, 0) circle (0.7pt);
			\draw [fill=black] (6.25, 0) circle (0.7pt);
			\draw [fill=black] (6.75, 0) circle (0.7pt);
			\draw (4.5, 0) -- (6, -3);
			\draw (5.5, 0) -- (6, -3);
			\draw (7.5, 0) -- (6, -3);
			\draw [fill=black] (2.75, 0) circle (0.7pt);
			\draw [fill=black] (3, 0) circle (0.7pt);
			\draw [fill=black] (3.25, 0) circle (0.7pt);
		\end{tikzpicture}
		
		\vspace{1em}
		\parbox{\textwidth}{\centering
			Figure 5. The graph $D_{n,k}$}
	\end{center}
	
	\medskip
	
	Proposition~\ref{prop5.1} motivates the following extremal problem.

\begin{prob}
	For $k\ge3$ and $n\ge k^2+2k$, determine $\operatorname{ex}^{\mathrm{HIST}}_k(n)$. Is
	$$
	\operatorname{ex}^{\mathrm{HIST}}_k(n)
	=\binom{n-k^2-k}{2}+k^2(k+1)?$$
	If so, is $D_{n,k}$ the unique extremal graph?
\end{prob}

\section*{Declaration of Competing Interest}
The authors declare that they have no known competing financial interests or personal relationships that could have appeared to influence the work reported in this paper.

\end{document}